\documentclass[reqno,oneside,12pt]{amsart}
\usepackage{amssymb,enumerate,bm}

\usepackage[T1]{fontenc}

\usepackage[width=14cm,centering]{geometry}

\newtheorem{theorem}{Theorem}
\newtheorem{thm}[theorem]{Theorem}

\newtheorem{lemma}[theorem]{Lemma}

\theoremstyle{definition}

\theoremstyle{remark}

\newtheorem{rem}{Remark}

\long\def\red#1{\bgroup
 \color{red}#1\egroup}

\newif\ifcnote\cnotefalse

\cnotetrue
\title[On differential polynomial rings over locally nilpotent rings]
    {On differential polynomial rings over locally nilpotent rings}

    \author[Chebotar]
{Mikhail Chebotar}



\address{Department of Mathematical Sciences,
Kent State University, Kent OH 44242, U.S.A.}
\email{chebotar@math.kent.edu}

\begin{document}
\begin{abstract}
Let $\delta$ be a derivation of a locally nilpotent ring $R$. Then the differential polynomial ring
$R[X; \delta]$ cannot be mapped onto a ring with a non-zero idempotent. This answers a recent question by
Greenfeld, Smoktunowicz and Ziembowski.
\end{abstract}
\maketitle

\section{Introduction}

Let $\delta:R\to R$ be a derivation of a ring $R$. By $R[X; \delta]$ we denote a differential polynomial ring and recall that the multiplication is defined by the condition $Xr = rX + \delta(r)$ for all $r \in  R$.

Recently, there has been a significant interest to the radical properties of differential polynomial rings \cite{BMS,GSZ,NZ,S,SZ}.
In particular, Smoktunowicz and Ziembowski \cite{SZ} proved that there exists a locally nilpotent
ring $R$ and a derivation $\delta$ of $R$ such that $R[X; \delta]$ is not Jacobson
radical, thus solving an open problem by Shestakov.

This paper is motivated by the following problem due to Greenfeld, Smoktunowicz and Ziembowski \cite[Question 6.5]{GSZ}:
Is there a locally nilpotent ring $R$ and a derivation $\delta$ such that
$R[X; \delta]$ maps onto a ring with a non-zero idempotent?

We will show that the answer to this question is negative, namely we will prove the following theorem:
\begin{thm}\label{T1}
Let $\delta$ be a derivation of the locally nilpotent ring $R$. Then the differential polynomial ring
$R[X; \delta]$ cannot be mapped onto a ring with a non-zero idempotent.
\end{thm}
 
Using the language of Radical Theory we could simply say that for a locally nilpotent ring $R$, the differential polynomial ring $R[X,\delta]$ is Behrens radical.

\section{Results}
Our approach will be based on the infinite-dimensional triangularization which was studied in the recent paper by Mesyan
\cite{M}.

Let $V$ be a vector space over a field $K$.
Following \cite{M} we will say that a transformation $t$ of a vector space $V$ is strictly triangularizable, if $V$ has a well-ordered basis such that $t$ sends each vector from that basis to the subspace spanned by basis vectors less than it.

Denote by End$_K(V)$ the $K$-algebra of all linear transformations of $V$.
We start with an obvious observation.

\begin{rem}\label{R1}
Let $K$ be a field, $V$ a nonzero $K$-vector space and $S$ a nilpotent
subalgebra of  End$_K(V)$. Then there exists a $1$-dimensional subspace
$W\subseteq V$ such that $S(W)=0$.
\end{rem}

\begin{proof}
Let $n$ be such a number that $S^n=0$, but $S^{n-1}\ne 0$, so there exist $s_1,\ldots, s_{n-1}\in S$ with
$s_1\cdots s_{n-1}\ne 0$. Let $v\in V$ be such a vector that $s_1\cdots s_{n-1}(v)=w\ne 0$.
Let $W$ be the linear span of $w$, then it is $1$-dimensional and $S(W)=0$.
\end{proof}
 
 We continue with the following useful remark (see \cite[p.19]{R} or \cite[Proposition 20]{M}).
 
 \begin{rem}\label{R2}
Let $K$ be a field, $V$ a nonzero $K$-vector space and $s$ a nilpotent
element of  End$_K(V)$. Then $s$ is strictly triangularizable.
\end{rem}

With Remark~\ref{R2} at hand we are ready to provide the following ``nilpotent analogy" of 
\cite[Theorem 15]{M}.

\begin{thm}\label{T2}
Let $K$ be a field, $V$ a nonzero $K$-vector space and $S$ a finite-dimensional nilpotent
subalgebra of  End$_K(V)$. Then there exists a well-ordered basis for $V$ with respect to which every element of
$S$ is strictly upper triangular.
\end{thm}
The proof is a word by word repetition of the proof of \cite[Theorem 15]{M} with the only exception:
Lemma 14 should be replaced by Remark~\ref{R1}.

We continue with a folklore result which is a corollary of the general Leibniz rule:
\begin{lemma}\label{L1}
Let $e$ and $x$ be elements of a ring $R$. Define $[e,x]_0=e$, $[e,x]_1=[e,x]=ex-xe$ and inductively $[e,x]_j=[[e,x]_{j-1},x]$ for $j>1$.
Then 
$$
ex^n=\sum_{j=0}^n \binom{n}{j} x^j[e,x]_{n-j},
$$
where $\binom{n}{j}$ are binomial coefficients.
\end{lemma}
 
We will also need the following technical result.
 
\begin{lemma}\label{L2}
Let $e$ and $x$ be elements of a ring $R$ with $e^2=e$. Then
for any non-negative integer $n$ we have 
$[e,x]_n=\sum_{i=0}^n r_i e [e,x]_i$
for some $r_i \in R$.  
\end{lemma}
\begin{proof}
The case $n=0$ is obvious as $[e,x]_0=e=e\cdot e[e,x]_0$ and we can take $r_0=e$.

We proceed by induction on $n\ge 1$. For $n=1$ we have 
$$
[e,x]=[e^2,x]=[e,x]e+e[e,x]=[e,x]e \cdot e+e \cdot e[e,x],
$$
so we may take $r_0=[e,x]$ and $r_1=e$. Suppose the statement is true for $n=k-1$.
Then using the Leibniz formula:
\begin{eqnarray*}
[e,x]_k=[[e,x],x]_{k-1}=[[e^2,x],x]_{k-1}=[[e,x]e+e[e,x],x]_{k-1}=\\
{}[[e,x]e,x]_{k-1}+[e[e,x],x]_{k-1}=\\
\sum_{j=0}^{k-1} \binom{k-1}{j} [e,x]_{j+1}[e,x]_{k-1-j}+
\sum_{j=0}^{k-1} \binom{k-1}{j} [e,x]_{k-1-j}[e,x]_{j+1}.
\end{eqnarray*}
The terms $[e,x]_ke=[e,x]_ke \cdot e$ and $e[e,x]_k=e\cdot e  [e,x]_k$ are of the desired form.
For the terms $[e,x]_{k-j}[e,x]_j$ where $j=1,\ldots,k-1$, we replace each $[e,x]_j$ using the inductive hypothesis
and consequently get the desired form as well.
\end{proof}

Now we are ready to state the key lemma:

\begin{lemma}\label{L3}
Let $N$ be a locally nilpotent subalgebra of End$_K(V)$. Suppose that there exist $a_0,a_1\ldots,a_n \in N$ and
$x\in$  End$_K(V)$ so that $a_0+xa_1+\ldots+x^na_n=e$ is an idempotent. Then $e=0$.
\end{lemma}

\begin{proof}
Let $S$ be a subalgebra of $N$ generated by $a_i$ and $[a_i,x]_j$, where $i=1,\ldots, n$ and $j=1,\ldots, n$.
Since $S$ is a finitely generated subalgebra of a locally nilpotent algebra, it must be nilpotent.
According to Theorem~\ref{T2}, $S$ is simultaneously strictly triangularizable in End$_{K}V$, so by \cite[Lemma 5]{M}  there exists a 
well-ordered (by inclusion) set of $S$-invariant subspaces $0=V_0\subseteq V_1\subseteq V_2\subseteq \ldots$ of
$V$, which is maximal as a well-ordered set of subspaces of $V$.

Since $S$ is nilpotent, we have that $S(V_i)\subseteq V_{i-1}$ for all positive integers $i$ and in particular
$S(V_1)=0$.

We claim that for any $k=0,\ldots, n$, and any positive integer $l$ we have $e[e,x]_k(V_l)=0$.

We proceed by induction on $l$. For $l=1$, we write 
$$[e,x]_k=[\sum_{i=0}^n x^ia_i,x]_k=\sum_{i=0}^n x^i[a_i,x]_k$$ 
and $[a_i,x]_k (V_1)=0$ since $S(V_1)=0$.
We established the basis of induction: $e[e,x]_k(V_1)=0$.

Suppose our claim is true for $l= {m-1}$. We want to prove the statement for $l=m$.
For any $v\in V_m$ we get
$$
[e,x]_k(v)=([a_0,x]_k+x[a_1x]_k+\ldots+x^n[a_n,x]_k)(v)=\sum_{i=0}^n x^i (u_i),
$$
where all $u_i \in V_{m-1}$, because $[a_i,x]_k\in S$. Using Lemma~\ref{L1} we obtain
$$
e[e,x]_k(v)=\sum_{i=0}^n ex^i (u_i)=\sum_{i=0}^n \sum_{j=0}^i \binom{i}{j} x^j[e,x]_{i-j}(u_i)
$$
and by Lemma~\ref{L2} each $[e,x]_{i-j}=\sum_{p=0}^{i-j} r_p e [e,x]_p$
for some $r_p \in R$.  
By the inductive hypothesis $e[e,x]_p(V_{m-1})=0$, so we get $[e,x]_{i-j}(u_i)=0$ and therefore
$e[e,x]_k(v)=0$.

As a consequence of our claim, the expression $e=a_0+xa_1+\ldots+x^na_n$ cannot be a non-zero idempotent, since for every $V_l$ we have $e(V_l)=0$.
\end{proof}

{\bf Proof of Theorem~\ref{T1}.}
Suppose that there exist a locally nilpotent ring $R$ and a derivation $\delta$ such that $R[X; \delta]$ can be mapped onto a ring with a non-zero idempotent. In other words, $R[X; \delta]$ is not Behrens radical and so there is a surjective homomorphism $\varphi$ of a ring $R[X; \delta]$  onto a subdirectly irreducible ring $A$ whose heart contains a non-zero idempotent $e$ \cite[Section 4.11]{GW}. Since the heart of $A$ contains a non-zero idempotent, $A$ is a prime ring and its extended centroid $K$ is a field \cite[Section 2.3]{BMM}. Denote by $Q$ the Martindale right ring of quotients of $A$.

Define the map $x: A\to A$ by the rule $x (\varphi(t))=\varphi(Xt)$ for all $t\in R[X; \delta]$. Since the ring $A$ is prime we claim that this map is well-defined. Indeed, suppose that $\varphi(t)=0$ and $\varphi(Xt)\ne 0$. By the primeness of $A$, there exists $t'\in R[X; \delta]$ with $\varphi(t')\varphi(Xt)\ne 0$. On the other hand, 
$$\varphi(t')\varphi(Xt)=\varphi(t'Xt)=\varphi(t'X)\varphi(t)=0,$$
a contradiction. It is clear that $x:A_A \to A_A$ is an endomorphism of a right $A$-module $A_A$ and so it is an element of $Q$.  Let $A'$ be a subring of $Q$ generated by $A$ and $x$, and let $R^{\#}$ be the ring $R$ with unity adjoined.
We define an additive map $\psi: R^{\#}[X; \delta]\to A'$ by the rule
$\psi(X^i)=x^i$ for all positive integers $i$, and $\psi(t)=\varphi(t)$ for all $t\in R[X; \delta]$.
By construction $\psi$ is a homomorphism that extends $\varphi$.
Now, a non-zero idempotent $e\in A\subseteq A'$ can be presented in the form
$$
e=\varphi(r_0+Xr_1+\ldots+X^nr_n)=\psi(r_0+Xr_1+\ldots+X^nr_n)=a_0+xa_1+\ldots+x^na_n,
$$
where
$\psi(r_i)=\varphi(r_i)=a_i$ and $\psi(X)=x$.

Let $D$ be a subring of $A'$ generated by $x,a_0,\ldots,a_n$ and let $B=D\cap \psi(R)$. Clearly, $B$ is a locally nilpotent ring and the subalgebra $BK$ of  $Q$ is also locally nilpotent. Since the subalgebra $DK$ of $A'K$ is finitely generated, it can be embedded into End$_K(V)$ for some $K$-vector space $V$ by \cite[Proposition 2.1]{GMM}.

Now we can assume that $x \in$  End$_K(V)$, $N=BK \subseteq$ End$_K(V)$ is locally nilpotent, and
$a_0+xa_1+\ldots+x^na_n=e \in$ End$_K(V)$ is a non-zero idempotent, so we can apply Lemma~\ref{L3}. 
However, by Lemma~\ref{L3} this idempotent $e$ must be zero, a contradiction.
The proof is thereby complete.

\end{document}